\documentclass[reqno]{amsart}
\usepackage[utf8x]{inputenc}
\usepackage{amssymb}
\usepackage[nocompress]{cite}
\newtheorem{theorem}{Theorem}[section]
\newtheorem{lemma}{Lemma}[section]

\newtheorem{remark}{Remark}
\theoremstyle{remark}

\newcommand{\eps}{\varepsilon}
\newcommand{\R}{\mathbb R}
\newcommand{\iR}{\int_{t=-R^2}}
\newcommand{\loc}{\mbox{\scriptsize{loc}}}
\newcommand{\supp}{\mbox{supp}}

\numberwithin{equation}{section}

\title{Regularity and long-time behavior of nonlocal heat flows}
\author{Stanley Snelson}
\address{Department of Mathematics, Courant Institute of Mathematical Sciences, New York University, 251 Mercer St., New York, NY 10012}
\email{snelson@cims.nyu,edu}
\date{March 14, 2015}
\begin{document}

\begin{abstract}
This article considers nonlocal heat flows into a singular target space. The problem is the parabolic analogue of a stationary problem that arises as the limit of a singularly perturbed elliptic system. It also provides a gradient flow approach to an optimal eigenvalue partition problem that, in light of the work of Caffarelli and Lin, is equivalent to a constrained harmonic mapping problem. In particular, we show that weak solutions of our heat equation converge as time approaches infinity to stationary solutions of this mapping problem. For these weak solutions, we also prove Lipschitz continuity and regularity of free interfaces, making use of a parabolic Almgren-type monotonicity formula.
% \subclass{35K55 \and 49N60 \and 35B40 \and 35B25}
\end{abstract}
\maketitle

\section{Introduction}\label{intro}
We consider a parabolic system motivated by the following problem, which arose in the study of shape optimization:
\begin{enumerate}
\item[\textbf{(P)}]\emph{Let $\Omega$ be a smooth, bounded domain in $\R^n$, and let $m\geq1$ be a positive integer. One seeks a partition of $\Omega$ into mutually disjoint subsets $\Omega_j$, $j=1,2,\ldots, m$, such that $\bar \Omega = \cup_{j=1}^m \bar \Omega_j$, and $\sum_{j=1}^m \lambda_1(\Omega_j)$ is minimized over all possible partitions of $\Omega$. Here $\lambda_1(\Omega_j)$ is the first Dirichlet eigenvalue of the Laplacian $\Delta$ on $\Omega_j$.}
\end{enumerate}
Shape optimization problems arise naturally in the field of structural design, when one seeks to optimize the design of a system governed by partial differential equations. Problem (P) has been studied by many authors; see for example \cite{BD, S, BZ, BBH} and the references in \cite{BBH}. Problem (P) also belongs to a broader class of partition problems that have attracted interest due to connections with spectral theory, see for example \cite{H, HHT, HHT2}.

To prove the existence of a minimizing partition, one must first define a class of admissible subsets of $\Omega$ so that the Dirichlet eigenvalues can be properly defined. Another more difficult issue is to introduce a topology on the space of partitions into admissible subsets so that the functional $\sum_{j=1}^m \lambda_1(\Omega_j)$ possesses an absolute minimum, and that the minimizer is a suitable partition of $\Omega$ into admissible subsets. An existence theorem for problem (P) was first established by Bucur-Buttazzo-Henrot \cite{BBH} in the context of so-called ``quasi-open'' partitions. Later, the work of Conti-Terracini-Verzini \cite{CTV0, CTVA, CTV3}, which considered (P) and various generalizations, developed in particular an existence proof for open partitions solving (P), making use of an equivalence between (P) and a constrained energy minimization problem for functions on $\Omega$. A similar equivalence was proven by Caffarelli and Lin \cite{CL2} for quasi-open partitions. We now describe their formulation of this equivalence, which is essential for the present study. Consider the constrained harmonic mapping problem
\begin{equation}\tag{\textbf{P*}}
\min \left\{ \int_\Omega |\nabla v|^2\,dx \,\Big|\, v\in H_0^1(\Omega,\Sigma), \int_\Omega v_j^2\,dx = 1 \mbox{ for } j = 1,\ldots,m \right\}
\end{equation}
where $\Sigma$ is the singular space
\[\Sigma = \left\{y\in\R^m\,\Big|\,\sum_{j\neq k} y_j^2y_k^2=0, y_k\geq 0 \text{ for } k = 1,\ldots,m\right\},\]
a union of $m$ perpendicular line segments meeting at the origin. Problems (P) and (P*) are equivalent in the following sense: there exists a solution $u$ of problem (P*) such that $u$ is Lipschitz in $\bar\Omega$, and such that the open sets $\Omega_j = \{x\in\Omega | u_j(x)>0\}$ for $j=1,\ldots,m$ form a partition of $\Omega$. That is, $\bar \Omega = \cup_{j=1}^m \bar\Omega_j$, and $\partial(\bar\Omega_j \setminus \Omega_j)$ are $C^{2,\beta}$ in $\Omega$ away from a relatively closed singular subset of $\Omega$, which is of Hausdorff dimension at most $n-2$ for $j=1,\ldots,m$. Furthermore, the partition of $\Omega$ given by $A_1=\bar\Omega_1$, $A_2=\bar\Omega_2\setminus A_1, \ldots,A_m = \bar\Omega_m\setminus\cup_{j=1}^{m-1} A_j$ is admissible in the sense defined in \cite{BBH}, and this partition solves problem (P). Conversely, for an optimal partition $A_1,\ldots,A_m$, let $u_j$ be the normalized eigenfunctions on $A_j$:
\[ \begin{cases}
\Delta u_j + \lambda_1(A_j) u_j = 0 &\mbox{in } \mathring A_j,\\
u_j=0 &\mbox{on } \bar\Omega\setminus \mathring A_j,
\end{cases}\]
with $u_j\in H_0^1(\Omega)$, $u_j\geq 0$, and $\int_\Omega|u_j|^2\,dx=1$. Then $u=(u_1,\ldots,u_m):\Omega\rightarrow\Sigma$ solves problem (P*). This equivalence leads to a proof of the existence of an optimal partition (a solution of problem (P)), see \cite{CL2} for the details.

Now that the optimal partition problem has been reformulated in a variational setting, it is natural to study the corresponding gradient flow. Because of the constraint $\int_\Omega |u_j|^2 = 1$ in problem (P*), the heat flow we are interested in should preserve the $L^2$ norm of each component. Another hint about the proper way to formulate a gradient flow corresponding to problem (P*) is provided by the body of work relating problems such as (P) to limiting problems of singularly perturbed elliptic systems (see \cite{CL3, CR, CLLL}). In particular, Conti-Terracini-Verzini \cite{CTV0, CTVA, CTV, CTV2} related singularly perturbed systems to optimal partition problems for nonlinear eigenvalues and Nehari's problem, and established Lipschitz continuity of the limiting solutions, as well as regularity of the free interfaces in two dimensions. Caffarelli-Lin \cite{CL3} studied the minimization problem
\begin{equation}\label{eq:stat}
\min \left\{ \int_\Omega \left[|\nabla u|^2+2F^\eps(u)\right] dx \,\Big|\, u\in H_0^1(\Omega,\R^m), u|_{\partial\Omega} = f\in\Sigma, f_j\geq 0\right\}
\end{equation}
where $F^\eps(u) = \sum_{1\leq i<j\leq m} \eps^{-2}(u_i)^2(u_j)^2$. In \cite{CL3}, the existence of minimizers $u^\eps$ of (\ref{eq:stat}) was established, and it was found that as $\eps\rightarrow 0$, the components of the limit of these minimizers are harmonic in the interiors of their supports, and that the supports become disjoint in the limit. (In particular, the limit maps into the space $\Sigma$ defined above. Note that $\Sigma=(F^\eps)^{-1}(0)\cap\{y_j\geq 0,j=1,\ldots,m\}$.) Lipschitz regularity of the limiting solution and free interface regularity for arbitrary target dimension were also established in \cite{CL3}, and problem (P) was seen as a special case of a larger class of limiting systems that also includes (\ref{eq:stat}). (Later, non-minimizing critical points of (\ref{eq:stat}) were analyzed by Tavares-Terracini \cite{TT} and Dancer-Wang-Zhang \cite{DWZ3}.)

It is therefore natural to look for a heat flow that preserves the $L^2$ energy of solutions and that arises as a limit to a singularly perturbed system analogous to the stationary problem (\ref{eq:stat}). In this regard, Caffarelli and Lin \cite{CL} considered the following nonlocal, nonlinear heat flow:
\begin{equation}\label{eq:epsilon}
\begin{cases}
\partial_tu_j^\eps-\Delta u_j^\eps = \lambda_j^\eps(t)u_j^\eps - \frac{\partial F^\eps( u^\eps)}{\partial u_j^\eps} & \mbox{in } \Omega \times \R_+, 1\leq j \leq m,\\
	    u^\eps(t,x) = 0 & \mbox{on } \partial\Omega\times\R_+,\\
	    u^\eps(0,x) = g(x) & \mbox{in } \Omega,
\end{cases}
\end{equation}
with the constraints
\[\int_\Omega|u_j^\eps|^2 dx = c_j^2, j = 1,\ldots,m. \]
Here $F^\eps(u) = \sum_{1\leq i<j\leq m} \eps^{-2}(u_i)^2(u_j)^2$ as above, $\lambda_j^\eps(t)$ is free, and
\begin{equation*}
\begin{cases}
 g\in H_0^1(\Omega,\Sigma), \,\,\,\, g_j(x)\geq 0 \mbox{ a.e. in } \Omega,\\
 c_j^2 = \int_\Omega g_j^2(x)dx > 0 ,\, 1 \leq j \leq m.
\end{cases}
\end{equation*}
Since the initial condition $g$ maps into $\Sigma$, each component $g_j$ is nonnegative, and the supports of $g_j$ are mutually disjoint. If a solution exists, it follows easily that 
\[\lambda_j^\eps(t) = \frac{1}{c_j^2}\int_\Omega\Big(\nabla |u_j^\eps|^2 + 2\sum_{i\neq j}\frac{(u_i^\eps)^2(u_j^\eps)^2}{\eps^2}\Big)\,dx.\]
Caffarelli-Lin showed that a unique strong solution $u^\eps$ exists and that the singular term in (\ref{eq:epsilon}) is controlled as $\eps\rightarrow 0$, so that $u^\eps$ converges in $H_{\loc}^1(\Omega\times\R_+)$, and the limit maps into $\Sigma$. 

We say that $u\in H_{\loc}^1(\Omega\times\R_+,\Sigma)$ such that $u(x,0)=g\in H_0^1(\Omega,\Sigma)\cap L^\infty(\Omega,\Sigma)$ and $u=0$ on $\partial\Omega\times\R_+$ is a \emph{suitable weak solution} of the constrained gradient flow for maps from $\Omega\rightarrow\Sigma$ if
\begin{equation}\label{eq:weak}
\begin{cases}
           \int_\Omega u_j^2 dx=\int_\Omega g_j^2 dx = c_j^2 >0, & \\
	   \partial_t u_j - \Delta u_j = \lambda_j(t)u_j-\nu_j(x,t) & \mbox{in } \mathcal{D}'(\Omega\times\R_+),\\
           u_j(\partial_t u_j - \Delta u_j) = \lambda_j(t)u_j^2 & \mbox{in } \mathcal{D}'(\Omega\times\R_+),\\
           \int_0^T\int_\Omega|\partial_t u|^2 dx dt + \int_\Omega|\nabla u|^2 dx \leq \int_\Omega |\nabla g|^2 dx, & 
\end{cases}
\end{equation}
for all $T>0$, $j=1,\ldots,m$, where $\lambda_j(t)=\frac{1}{c_j^2}\int_\Omega|\nabla u_j|^2 dx$, and $\nu_j(x,t)$ are nonnegative Radon measures supported in $\{u_j(x,t)=0\}$. 
In \cite{CL} it is shown, using various a~priori estimates, that the solution $u^\eps$ of (\ref{eq:epsilon}) converges strongly in $H_{\loc}^1(\Omega\times\R_+)$ to a suitable weak solution as in (\ref{eq:weak}).

The goal of the present article is threefold: First, to prove Lipschitz continuity of suitable weak solutions $u$ as in (\ref{eq:weak}). Second, to study the regularity of the interface between the supports of the components $u_j$. The picture we have in mind of the interface is similar to that in \cite{CL3} for the limit of the stationary problem (\ref{eq:stat}):  a smooth set that locally separates two subdomains of $\Omega$, and a lower-dimensional singular set where three or more subdomains come together. Finally, we would like to show that, as $t\rightarrow\infty$, suitable weak solutions $u$ converge strongly in $H_0^1(\Omega)$ to maps into $\Sigma$ that are stationary solutions of problem (P*). As expected, the components of these stationary solutions are eigenfunctions on their mutually disjoint supports. The Lipschitz continuity of these eigenfunctions and regularity of the corresponding nodal sets were established in \cite{CTV3} and \cite{HHT}. (See also the papers \cite{TT, DWZ3} regarding the regularity of stationary solutions to minimization problems like (\ref{eq:stat}) and the corresponding nodal partitions.)

In \cite{DWZ1}, Dancer et. al. proved a uniform H\"older estimate for a class of parabolic systems that are closely related to (\ref{eq:epsilon}) (but which lack the nonlocal term $\lambda_j^\eps(t)$), using a blow-up method and a monotonicity formula of Almgren type, first introduced by Poon \cite{P}. The present paper uses Poon's formula to obtain spatial Lipschitz regularity for (\ref{eq:weak}), the $\eps\rightarrow 0$ limit of (\ref{eq:epsilon}), but our estimate is local rather than global.

Another type of singularly perturbed system, which is related to the Lotka-Volterra model from population dynamics, features a term like $\sum_{j\neq i} \eps^{-2}u_i^\eps u_j^\eps$ in place of the $\sum_{j\neq i} \eps^{-2}u_i^\eps (u_j^\eps)^2$ term found in (\ref{eq:stat}) and (\ref{eq:epsilon}). Such a system was treated in \cite{CKL} in both the elliptic and parabolic versions, and using a viscosity method, the authors obtained local Lipschitz continuity and free boundary regularity results for the limiting solutions similar in spirit to those proved in this paper for (\ref{eq:weak}). A similar system of this type was studied by Dancer et. al. \cite{DWZ2}, who characterized the $\eps\rightarrow 0$ limit in terms of differential inequalities. For $\eps>0$, systems of this type do not have a variational structure as (\ref{eq:epsilon}) does (as a result, substantially different methods are required), but surprisingly, the singular limit exhibits a variational structure. As in the present study, the limiting system in the parabolic case converges to a stationary solution of the corresponding varational problem, and in fact, similar behavior is seen for the $\eps$ system when $\eps$ is small; see \cite{DWZ2}.

The gradient flow of harmonic maps into nonpositively curved metric spaces (such as $\Sigma$) was also studied by Mayer \cite{M} using a finite difference approach, and his methods could be generalized to the setting in \cite{CL}. However, the gradient flows constructed in \cite{M} are not convenient to work with when trying to establish Lipschitz continuity and free boundary regularity. The methods in \cite{CL} and this paper could be applied to some of the more general classes of parabolic systems discussed above, as well as singularly perturbed systems whose elliptic versions were studied in \cite{CL3, CR}. 

The paper is organized as follows. In Sect. \ref{sec:LC}, we derive a monotonicity formula of Almgren type and apply it to prove Lipschitz continuity of suitable weak solutions in Theorem \ref{thm:Lip}. Sect. \ref{sec:Bdry} is devoted to interface regularity between the various components $u_j$. Theorems \ref{thm:dim} and \ref{thm:Reg} characterize this interface as the union of a $C^{2,\beta}$ hypersurface and a singular set of codimension at least $2$. Theorem \ref{thm:dim}, which bounds the dimension of the singular set, also makes use of the monotonicity formula (Lemma~\ref{lem:mono}). The argument in Sect. \ref{sec:Bdry} follows a similar outline to the proofs of free boundary regularity for stationary problems in \cite{CL2} and \cite{CL3}, but the time-dependence and nonlocality introduce some new difficulties. In Sect. \ref{sec:time} we look at the long-time behavior of $u(x,t)$ and show that $u$ converges to a stationary solution of problem (P*) in Theorem \ref{thm:probPstar}.

The author would like to thank his advisor Fanghua Lin for suggesting this problem and for insightful discussions.

\section{Lipschitz continuity}\label{sec:LC}
%The parabolic ball $Q_r(x_0,t_0)$ is defined by $Q_r(x_0,t_0)=\{|x-x_0|<r,t_0-r^2\leq t \leq t_0 \}$.
The Lipschitz continuity of weak solutions to (\ref{eq:weak}) relies on a monotonicity formula of Almgren type. Almgren \cite{A} showed that if $v$ is a harmonic function on $B_1$ such that $v(0)=0$, then the functional $rD(r)/L(r)$ is nondecreasing in $r$, where $D(r) = \int_{B_r}|\nabla v|^2\,dx$ and $L(r) = \int_{\partial B_r} v^2\,d\sigma$, and used this formula to study the growth of harmonic functions. These methods were later generalized to solutions of general second-order elliptic equations, see e.g. \cite{GL}. Almgren-type monotonicity for linear parabolic equations was considered by Poon \cite{P}, and his parabolic monotonicity formula has since been applied in a variety of settings, for example in studying the regularity of a parabolic obstacle problem by Danielli et. al. \cite{D}, and in the works of Dancer et. al. \cite{DWZ1, DWZ3} on segregating species models, as discussed in the Introduction.

Denote by $G_{x_0,t_0}(x,t)$ the Green's function for the backwards heat equation at a point $(x_0,t_0)$:
\[G_{x_0,t_0}(x,t)=\frac{1}{(4\pi(t_0-t))^{n/2}}\exp\left(-\frac{|x-x_0|^2}{4(t_0-t)}\right).\]
Recall that $\partial_t G_{x_0,t_0} = -\Delta G_{x_0,t_0}$ and $\nabla G_{x_0,t_0} = \dfrac{x-x_0}{2(t-t_0)}\,G_{x_0,t_0}$. Let $G=G_{0,0}$. Following Poon \cite{P}, for a function $v:\Omega\times\R_+\rightarrow\Sigma$ and a point $(x_0,t_0)$, we define the functionals
\begin{align*}
I(v,x_0,t_0,R) &= R^2\int_{t=t_0-R^2} |\nabla v(x,t)|^2 G_{x_0,t_0}(x,t)\,dx, \\
 H(v,x_0,t_0,R) &= \int_{t=t_0-R^2} d_\Sigma^2(v(x,t),v(x_0,t_0)) G_{x_0,t_0}(x,t)\,dx,\\
N(u,x_0,t_0,R) &= \frac{2I(u,x_0,t_0,R)}{H(u,x_0,t_0,R)},
\end{align*}
where the integrals are taken over the time slice $\Omega\times \{t_0-R^2\}$. The distance in the singular space $\Sigma$ is given by $d_\Sigma(p,q) = |p-q|$ if $p$ and $q$ are in the same line of $\Sigma$, and $d_\Sigma(p,q) = |p| + |q|$ otherwise. For functions $v$ that map into $\R^m$, we will slightly abuse the above notation by writing $H(v,x_0,t_0,R) = \int_{t=t_0-R^2} |v(x,t)-v(x_0,t_0)|^2G\,dx$. We will often write $N(u,R)$ or $N(R)$, etc., when the dependence on $u$ and/or $(x_0,t_0)$ is clear. 

For a solution $w$ of the vector heat equation $w_t=\Delta w$, it can be shown (see \cite{P}) that $N(w,R)$ is monotonically increasing in $R$. In our situation, with $u$ a suitable weak solution as in (\ref{eq:weak}), we can only obtain that $N(u,R)$ plus a monomial term is increasing, but this modified monotonicity will be sufficient to prove that $u$ is Lipschitz.

\begin{lemma}\label{lem:mono}
For $u$ a suitable weak solution of \textup{(\ref{eq:weak})} and $(x_0,t_0)\in\Omega\times\R_+$, there exists a constant $C$ independent of $(x_0,t_0)$ such that the quantity 
\[N(u,x_0,t_0,R) + CR^4\]
is nondecreasing in $R$ for $0<R<R_0$, where $R_0$ is the parabolic distance from $(x_0,t_0)$ to the parabolic boundary $\partial_P(\Omega\times\R_+)$.
\end{lemma}
\begin{remark}
Since $u$ is a Sobolev function, $N(u,R)$ is defined on $(0,R_0)$ only up to a set of measure zero. We interpret the monotonicity of $N(u,R)+CR^4$ as follows: there exists a nondecreasing function that agrees with $N(u,R)+CR^4$ almost everywhere on $(0,R_0)$.
\end{remark}
\begin{proof}
By translation, we may assume our equations are defined on $\Omega\times[-R_0^2,\infty)$, and that $(x_0,t_0)=(0,0)$ and $u(0,0)=0$. We will work with the solution $u^\eps$ of the approximating system (\ref{eq:epsilon}) and study the derivative in $R$ of
 \[
 N^\eps( u^\eps,0,0,R) = \frac{2R^2\int_{t=-R^2}\left[|\nabla  u^\eps|^2+2F^\eps( u^\eps)\right]\, dx } { \int_{t=-R^2} | u^\eps|^2G\,dx}.
  \]
It will be convenient to write (\ref{eq:epsilon}) in the vector form
\begin{equation}\label{eq:vec}
\partial_t u^\eps - \Delta u^\eps = \lambda^\eps\circ u^\eps - \nabla F^\eps(u^\eps),\end{equation}
where $\lambda^\eps(t)$ is the vector with $j$-th component $\lambda_j^\eps(t)$, and $\circ$ denotes the entrywise product. Here, the gradient of $F^\eps$ is taken with respect to $u^\eps\in\R^m$.

To simplify the notation, we set $u= u^\eps$. For $0<R<R_0$, we will now define 
$I^\eps(R)=R^2\iR\left[|\nabla u|^2+2F^\eps( u)\right]G\,dx$,
so that $N^\eps=2I^\eps/H$. Integrating by parts and using (\ref{eq:vec}) and  and $\nabla G = \frac{x}{2t}G$, we have
\begin{align}\label{eq:IR}
I^\eps(R) &= R^2\iR\left[|\nabla u|^2+2F^\eps(u)\right]G\,dx\nonumber\\
     &= -R^2\iR \left[u\cdot\left(u_t-\lambda^\eps(t)\circ u + \nabla F^\eps(u) + \nabla u\cdot\frac{x}{2t}\right) - 2F^\eps(u)\right]G\,dx\nonumber\\
     &= -R^2\iR \left[u\cdot\left(u_t-\lambda^\eps(t)\circ u + \nabla u\cdot\frac{x}{2t} \right)+ 2F^\eps(u)\right]G\,dx.
\end{align}
The last equality comes from the fact that
\[u\cdot\nabla F^\eps(u) = \sum_{i=1}^m u_i\frac{1}{\eps^2}\sum_{j\neq i} 2u_iu_j^2 = 4\sum_{i<j}\frac{1}{\eps^2}(u_iu_j)^2 = 4F^\eps(u).\]
Using $G_t=-\Delta G$, we differentiate $H(R)$:
\begin{align}\label{eq:Hpr}
 H'(R) &= -2R\iR \left[2u\cdot u_t G + |u|^2G_t\right]\,dx\nonumber\\
       &= -4R\iR u\cdot\left(u_t+\nabla u\cdot\frac{x}{2t}\right)G\,dx.
\end{align}

We will need the following consequence of (\ref{eq:IR}) and (\ref{eq:Hpr}) in the proof of Theorem~\ref{thm:Lip}:
\begin{equation}\label{eq:Hp}
H'(R) = \frac{4}{R}\left(I^\eps(R) - R^2\iR \left[u\cdot(\lambda^\eps(t)\circ u) - 2F^\eps(u)\right] G\,dx\right).
\end{equation}

We now rewrite (\ref{eq:IR}) and (\ref{eq:Hpr}) as
\begin{align*}
 I^\eps(R) &= -R^2\iR u\cdot\Big(u_t + \nabla u\cdot\frac{x}{2t} -\frac{1}{2}\lambda^\eps(t)\circ u -\frac{1}{2}\lambda^\eps(t)\circ u\Big)G\,dx\\
 &\quad -R^2\iR 2F^\eps(u)G\,dx, \\
 H'(R) &= -4R\iR u\cdot\Big(u_t+ \nabla u\cdot\frac{x}{2t}
-\frac{1}{2}\lambda^\eps(t)\circ u +\frac{1}{2}\lambda^\eps(t)\circ u\Big)G\,dx. 
\end{align*}
Thus,
\begin{align}\label{eq:HI}
 H'(R)I^\eps(R) &= 4R^3\left(\iR u\cdot\Big(u_t + \nabla u\cdot\frac{x}{2t} -\frac{1}{2}\lambda^\eps(t)\circ u\Big)\,G\,dx\right)^2\nonumber\\
           &\quad - R^3\left(\iR u\cdot(\lambda^\eps(t)\circ u) \,G\,dx\right)^2\nonumber\\
           &\quad + 4R^3\left(\iR u\cdot\left(u_t+\nabla u\cdot\frac{x}{2t} \right)G\,dx\right)\left(\iR 2F^\eps(u)G\,dx\right).
\end{align}

Next, we differentiate $I^\eps(R)$, using the properties of the heat kernel:
\begin{align*}
 (I^\eps)'(R) &=  2R\iR \left[|\nabla u|^2 + 2F^\eps(u)\right]G\,dx \\
       &- 2R^3\iR \left[2\nabla u:\nabla u_t\,G - |\nabla u|^2\Delta G + 2\nabla F^\eps(u)\cdot u_tG - 2F^\eps(u)\Delta G\right]dx,
\end{align*}
where $A:B=\mbox{tr}(A^t B)$. Integrating by parts and using $\nabla G = \dfrac{x}{2t}G$ again, we obtain
\begin{align*}       
 (I^\eps)'(R) &= 2R\iR \left[|\nabla u|^2 + 2F^\eps(u)\right]G\,dx\\
  & - 4R^3\iR \Big[\nabla u:\nabla u_t + \sum_{j=1}^m \nabla u_j\cdot D^2u_j\cdot\frac{x}{2t}\\ 
  &\hspace{3cm} + \nabla F^\eps(u)\cdot u_t + \nabla F^\eps(u)\cdot\nabla u\cdot\frac{x}{2t}\Big] \,G\,dx\\
       &= 2R\iR 2F^\eps(u)G\,dx\\
       & -4R^3\iR \left[\nabla u : \nabla\left(u_t+\nabla u\cdot\frac{x}{2t}\right) + \nabla F^\eps(u)\cdot \left(u_t+\nabla u\cdot\frac{x}{2t}\right)\right]G\,dx.
\end{align*}
The term $2R\int 2F^\eps(u)G\,dx\geq 0$ will not be needed in the final estimate. Integrating by parts again and using (\ref{eq:vec}), we have
\begin{align}\label{eq:Ip}
 (I^\eps)'(R) &\geq 4R^3\iR \left(\Delta u + \nabla u\cdot\frac{x}{2t} - \nabla F^\eps(u)\right)\cdot\left(u_t+\nabla u\cdot\frac{x}{2t}\right)G\,dx\nonumber\\
       &= 4R^3\iR \left(u_t+\nabla u\cdot\frac{x}{2t}-\lambda^\eps(t)\circ u \right)\cdot\left(u_t+\nabla u\cdot\frac{x}{2t}\right)G\,dx \nonumber \\      
       &= 4R^3\iR \Big|u_t+\nabla u\cdot\frac{x}{2t} - \frac{1}{2}\lambda^\eps(t)\circ u\Big|^2G\,dx\nonumber \\       
       &\quad - R^3\iR |\lambda^\eps(t)\circ u|^2G\,dx. 
\end{align}

Combining (\ref{eq:HI}) and (\ref{eq:Ip}), we get
\begin{align}\label{eq:IHHI}
 &(I^\eps)'(R)H(R) - H'(R)I^\eps(R)\nonumber\\
 &\qquad \geq 4R^3\left(\iR \Big|u_t+\nabla u\cdot\frac{x}{2t} - \frac{1}{2}\lambda^\eps(t)\circ u\Big|^2G\,dx\right)\left(\iR |u|^2G\,dx\right)\nonumber\\
 &\qquad \quad - R^3\left(\iR|\lambda^\eps(t)\circ u|^2G\,dx\right)\left( \iR |u|^2G\,dx\right) \nonumber\\
  &\qquad \quad - 4R^3\left(\iR u\cdot\Big(u_t + \nabla u\cdot\frac{x}{2t} -\frac{1}{2}\lambda^\eps(t)\circ u\Big)\,G\,dx\right)^2\nonumber\\
           &\qquad\quad + R^3\left(\iR u\cdot(\lambda^\eps(t)\circ u) \,G\,dx\right)^2\nonumber\\
           &\qquad\quad - 4R^3\left(\iR u\cdot\left(u_t+\nabla u\cdot \frac{x}{2t}\right)G\,dx\right)\left(\iR 2F^\eps(u)G\,dx\right).
\end{align}
By the Cauchy-Schwarz inequality,
\begin{align}\label{eq:4R3}
 & 4R^3\left(\iR u\cdot\Big(u_t + \nabla u\cdot\frac{x}{2t} -\frac{1}{2}\lambda^\eps(t)\circ u\Big)\,G\,dx\right)^2\nonumber\\
 &\quad\leq 4R^3\left(\iR \Big|u_t+\nabla u\cdot\frac{x}{2t} - \frac{1}{2}\lambda^\eps(t)\circ u\Big|^2G\,dx\right)\left(\iR |u|^2G\,dx\right).
\end{align}
It can be shown that $\sum_{j=1}^m |\lambda_j^\eps(t)|\leq M_0$ for some $M_0$ uniform in $t$ and $\eps$. (See \cite[Remark~3.2]{CL}.) Thus,
\begin{equation}\label{eq:R3}
- R^3\left(\iR|\lambda^\eps(t)\circ u|^2G\,dx\right)\left( \iR |u|^2G\,dx\right) \geq -M_0^2R^3(H(R))^2.
\end{equation}
Noting that, by (\ref{eq:Hpr}), $-4R^3\iR u\cdot\left(u_t+\nabla u\cdot \frac{x}{2t}\right)G\,dx = R^2H'(R)$, we conclude from (\ref{eq:IHHI}), (\ref{eq:4R3}), and (\ref{eq:R3}) that
\begin{equation*}
(N^\eps)'(R) \geq -2M_0^2R^3 + \frac{2R^2H'(R)}{(H(R))^2}\iR 2F^\eps(u)G\,dx.
\end{equation*}
It is clear from (\ref{eq:Hp}) that $H'(R)\geq -4R\iR u\cdot (\lambda^\eps(t)\circ u)G\,dx\geq -4M_0R H(R)$, which gives
\begin{equation}\label{eq:ineq}
(N^\eps)'(R) \geq -2M_0^2R^3 - \frac{8R^3}{H(R)}\iR 2F^\eps(u)G\,dx.
\end{equation}

To pass to the limit as $\eps\rightarrow 0$, we must first integrate in $R$. This is because we only have convergence of $u^\eps$ to $u$ in $H_0^1(\Omega\times\R_+)$, and (\ref{eq:ineq}) holds only for a time slice. For $0<R_1<R_2<R_0$, we have
 \[N^\eps(R_2)-N^\eps(R_1) \geq -C(R_2^4-R_1^4) - \int_{R_1}^{R_2}\frac{8R^3}{H(R)}\iR 2F^\eps(u)G\,dx\,dR.\]
In \cite[Corollary~3.4]{CL}, it is established that as $\eps\rightarrow 0$, 
\begin{equation}\label{eq:Feps}
\int_{t_1}^{t_2} \int_\Omega F^\eps (u^\eps)dxdt\rightarrow 0
\end{equation}
 for any $t_1<t_2$. Because $\int_\Omega |u^\eps(x,t)|^2\,dx=1$ for all $t$, we can bound $H(u^\eps,R)$ from below on the interval $[R_1,R_2]$, uniformly in $\eps$. This implies 
 \begin{equation}\label{eq:limeps}
 \lim_{\eps\rightarrow 0}\, [N^\eps(u^\eps,R_2)-N^\eps(u^\eps,R_1)] \geq -C(R_2^4-R_1^4).
 \end{equation}
It remains to show that the left hand side of (\ref{eq:limeps}) converges to $N(u,R_2)-N(u,R_1)$, where $u$ is a suitable weak solution of (\ref{eq:weak}). By the strong convergence of $u^\eps$ to $u$ in $H_{\loc}^1\left(\Omega\times[-R_0^2,\infty)\right)$, it follows that $\int_{R_1}^{R_2} H(u^\eps,R)\,dR\rightarrow\int_{R_1}^{R_2} H(u,R)\,dR$ and that $\int_{R_1}^{R_2} I^\eps(u^\eps,R)\,dR\rightarrow\int_{R_1}^{R_2} I(u,R)\,dR$. Letting $t_i = -R_i^2, i=1,2$, we see from (\ref{eq:Hpr}) that
\begin{align*}
\int_{R_1}^{R_2}\frac{d}{dR}H(u^\eps,R)\,dR &= 2\int_{t_1}^{t_2}\int_\Omega u^\eps\cdot \left(u^\eps_t+\nabla u^\eps\cdot\frac{x}{2t}\right)G\,dx\,dt\\
&\rightarrow -4\int_{R_1}^{R_2}R\iR u\cdot\left(u_t+\nabla u\cdot\frac{x}{2t}\right)G\,dx\\
&=\int_{R_1}^{R_2}\frac{d}{dR} H(u,R)\,dR.
\end{align*}
Similarly, we can conclude from (\ref{eq:Ip}) and (\ref{eq:limeps}) that $\int_{R_1}^{R_2}\frac{d}{dR}I^\eps(u^\eps,R)\,dR$ converges to $\int_{R_1}^{R_2}\frac{d}{dR}I(u,R)\,dR$. Noting again that $H(u^\eps,R)$ is uniformly bounded from below for $R\in[R_1,R_2]$, and that $H$ and $I$ are uniformly bounded from above, it is straightforward to check that
\[\int_{R_1}^{R_2}\frac{d}{dR}\left(\frac{2I^\eps(u^\eps,R)}{H(u^\eps,R)}\right)dR\rightarrow\int_{R_1}^{R_2}2\frac{I'(u,R)H(u,R)-I(u,R)H'(u,R)}{(H(u,R))^2}\,dR,\]
and we obtain $N(u,R_2)+CR_2^4\geq N(u,R_1)+CR_1^4$.

\end{proof}

By Lemma \ref{lem:mono},  the \emph{frequency} $N(u,x_0,t_0):=\lim_{R\rightarrow0} N(u,x_0,t_0,R)$ exists, and we have $N(u,x_0,t_0)~\leq~N(u,x_0,t_0,R) +~CR^4$.  Sometimes we will write $N(x_0,t_0)$ when the dependence on $u$ is clear. When $u(x_0,t_0)=0$, the frequency $N(u,x_0,t_0)$ is equal to the parabolic vanishing order of $u$ at the point $(x_0,t_0)$. In other words, if as $(x,t)\rightarrow (x_0,t_0)$, $u(x,t)$ is in the class $O(\sqrt{|x-x_0|^{2k}+|t-t_0|^k})$ but the quantity $u(x,t)/\sqrt{|x-x_0|^{2(k-1)}+|t-t_0|^{(k-1)}}\rightarrow 0$, then $N(u,x_0,t_0) = k$. We will also use the fact that for a solution of $v_t=\Delta v$, the frequency exists, and $N(v,x,t)$ gives the degree of the first term in the Taylor expansion of $v$ in caloric polynomials around $(x,t)$. Since $N(x,t,R)$ is continuous in $x$ and $t$, and $N(x,t)\leq N(x,t,R)+CR^4$, we have that $N(x,t)$ is upper-semicontinuous in $(x,t)$.

Let $\Gamma= \{(x,t)\in\Omega\times\R_+ | u(x,t)=0\}$. Because $u$ maps into $\Sigma$, the free interface $\Gamma$ divides $\Omega\times\R_+$ into $m$ mutually disjoint sets $\{u_j(x,t)>0\}$, $j=1,\ldots,m$.

\begin{theorem}\label{thm:Lip}
 A suitable weak solution $u(x,t)$ of \textup{(\ref{eq:weak})} is locally Lipschitz in $x$. 
\end{theorem}
\begin{proof}

On the interior of each set $\{u_j(x,t)>0\}$, we see from (\ref{eq:weak}) that $u_j$ satisfies $\partial_t u_j-\Delta u_j=\lambda_j(t)u_j$ in $\mathcal{D}'(\{u_j(x,t)>0\})$. Since $\lambda_j(t)$ is in $L^\infty(\R_+)$, the standard parabolic estimates (see, e.g., \cite[Chapter~7]{Lie}) imply that $u$ is locally Lipschitz in $x$ on the interior of $\{u_j(x,t)>0\}$ for each $j=1,\ldots,m$.

We will now use Lemma \ref{lem:mono} to show that $u$ grows linearly from the free interface. Because $u\in H_{\loc}^1(\Omega\times\R_+)$, we have that $u$ is $L^2$-approximately differentiable in $x$, that is, $|B(x,r)|^{-1}\int_{B(x,r)} r^{-2}|u(y,t)-u(x,t)-\nabla u(x,t)(y-x)|^2\,dy \rightarrow 0$ as $r\rightarrow 0$, for $\mathcal L^{n+1}$-a.e. $(x,t)\in\Omega\times\R_+$. (See \cite{L}, page 36.) At any such point where $u$ is approximately differentiable, the frequency $N(x,t)\geq 1$. By upper-semicontinuity, $N(x,t)\geq 1$ everywhere in $\Omega\times\R_+$.
 
Fix a point $(x_0,t_0)\in\Gamma$. As above, let $R_0$ be the distance from $(x_0,t_0)$ to the parabolic boundary of $\Omega\times\R_+$. For the approximating heat flow $u^\eps(x,t)$, we conclude from (\ref{eq:Hp}) that  
  \[\frac{d}{dR} H(u^\eps,R) \geq \frac{4}{R}\left( I^\eps (u^\eps,R) - R^2M_0\int_{t=t_0-R^2} |u^\eps|^2 G_{x_0,t_0}\right) dx,\]
for $0<R<R_0$. This implies
\begin{equation}\label{eq:logH}
\frac{d}{dR}\log H(u^\eps,R) \geq \frac{2}{R}N^\eps(u^\eps,R)-CR.
\end{equation}
From the proof of Lemma \ref{lem:mono}, it is clear that $\int_{R_1}^{R_2}N^\eps(u^\eps,R)\,dR\rightarrow\int_{R_1}^{R_2}N(u,R)\,dR$ for $0<R_1<R_2\leq R_0$. Then, integrating (\ref{eq:logH}) from $R_1$ to $R_2$ and taking the limit as $\eps\rightarrow 0$, we obtain
\begin{align*}
\log \frac{H(u,R_2)}{H(u,R_1)} &\geq \int_{R_1}^{R_2}\frac{2}{R}N(u,R)\,dR - C(R_2^2-R_1^2)\\
 &\geq \int_{R_1}^{R_2} \frac{2}{R}\left(N(u,x_0,t_0)-CR^4\right)\,dR - C(R_2^2-R_1^2).
\end{align*}
with $C$ independent of $(x_0,t_0)$. Since $N(u,x_0,t_0)\geq 1$, we conclude
\[H(u,R_1)\leq H(u,R_2)\frac{R_1^2}{R_2^2}\exp\left(C(R_2^4-R_1^4+R_2^2-R_1^2)\right).\]
Choosing $R_2=R_0$ and renaming $R_1=R$, this implies
\begin{equation*}
\int_\Omega d_\Sigma^2\left(u(x,t_0-R^2),0\right)G_{x_0,t_0}(x,t_0-R^2)\,dx = H(u,R)\leq C_0R^2.
\end{equation*}
for $0<R\leq R_0$, where $C_0$ depends on $R_0$.
Making the change of variables $x=x_0+Ry$, we obtain
\begin{equation}\label{eq:Hineq}
\frac{1}{R^2}\int_{\R^n}d_\Sigma^2\left(u(x_0+Ry,t_0-R^2),0\right)G_{0,0}(y,1)\,dy \leq C_0,
\end{equation}
where we have extended $u(x,t)$ by $0$ for $x\notin\Omega$.

Let $(x_1,t_1)\in\Omega\times\R_+$ be such that $u(x_1,t_1)\neq 0$ and $t_1<t_0$, and suppose that $(x_0,t_0)$ is the closest point in $\Gamma$ to $(x_1,t_1)$ in the parabolic distance. Let $s = t_0-t_1$, and let $\rho = d_P((x_0,t_0),(x_1,t_1))=\sqrt{s+|x_0-x_1|^2}$. Since $u$ maps into $\Sigma$, we have $u_{j_0}(x_1,t_1)>0$ for some $1\leq j_0\leq m$, and thus $u_{j_0}>0$ in $Q=B_{\rho/2}(x_1)\times (t_1-\rho^2/4,t_0)$ and $u_j=0$ in $Q$ for all other $j$. By (\ref{eq:weak}), $u_{j_0}$ satisfies 
\[\partial_t u_{j_0}-\Delta u_{j_0} =\lambda_{j_0}(t)u_{j_0}\geq 0 \quad \mbox{ in } \mathcal{D}'(Q).\]
Assume that for some point $(x_2,t_1)\in Q$, there holds
\[u_{j_0}(x_2,t_1) \geq Mr\]
with $r=\rho/2$ and $M$ large. Define the cylinders $Q_1= B_r(x_1)\times (t_1-r^2,t_1+s/4)$ and $Q_2=B_r(x_1)\times (t_0-s/4,t_0)$. By the parabolic Harnack inequality,
\[ \inf_{Q_2} u_{j_0} \geq C_1 \sup_{Q_1} u_{j_0} \geq C_1 Mr.\]
We observe $d_\Sigma(u(x,t),0)=|u_{j_0}(x,t)|$. It is clear that $s=t_0-t_1=a^2r^2$  for some fixed $a\leq\frac{1}{2}$. Taking $R=ar/4$ in (\ref{eq:Hineq}), so that $R^2<s/4$, we see
\[C_0\geq \int_{\R^n}\frac{\left|u_{j_0}(x_0+ary/4,t_0-a^2r^2/16)\right|^2}{(ar/4)^2}G_{0,0}(y,1)\,dy \geq c(n)C_1 M,\]
a contradiction if $M>C_0/(c(n)C_1)$.
\end{proof}

\begin{lemma}\label{lem:Leb}
 The free interface $\Gamma\subset\Omega\times\R_+$ has zero $(n+1)$-dimensional Lebesgue measure.
\end{lemma}
\begin{proof}
The lemma follows from a much stronger result controlling the size of the nodal set of solutions to (\ref{eq:weak}) that could be concluded from the Almgren monotonicity formula, see \cite{Li} for the details. 
\end{proof}

\section{Regularity of free interfaces}\label{sec:Bdry}
We now study the regularity of the free interface $\Gamma = \{(x,t)\in\Omega \times \R_+ | u(x,t) = 0\}$, where $u$ is a suitable weak solution of (\ref{eq:weak}). 

First, we want to establish a gap in the values of the frequency function $N(x,t)$ for $(x,t)\in\Gamma$. As noted above, the smallest possible value of $N(x,t)$ is 1, corresponding to degree-$1$-homogeneous behavior of each component of $u$ near $(x,t)$. We will show that the next possible value is $N(x,t) = 1+\delta_n$ for some positive constant $\delta_n$ depending on the dimension. For this, we need the homogeneous blow-ups of $u$ for points in $\Gamma$:
 
\begin{lemma}\label{lem:blowup}
Assume $(x_0,t_0)\in\Gamma$. Then for any $\rho_k \rightarrow 0$, there is a subsequence of $u_{\rho_k}=u(x_0+\rho_k x, t_0 + \rho_k^2 t)$ that converges weakly in $H_{\loc}^1(\R^n\times\R_-)$ to a function $\bar u(x,t)$ defined on $\R^n\times\R_-$.  Moreover, $\bar u$ satisfies $\bar u_t-\Delta \bar u=0$ in $\mathcal{D}'(\R^n\times\R_-)$ and $\bar u(\rho x, \rho^2 t) = \rho^\alpha\bar u(x,t)$ for $\rho\geq 0$, where $\alpha=N(x_0,t_0)$.
\end{lemma}
 \begin{proof}
For any $\rho\geq 0$, it is straightforward to check that $u_\rho = u(x_0+\rho x,t_0+\rho^2 t)$ satisfies $N(u_\rho,0,0,r) = N(u,x_0,t_0,\rho r)$. Thus, as $\rho_k\rightarrow 0$, $u_{\rho_k}$ are uniformly bounded in $H_{\loc}^1(\R^n\times\R_-)$ and converge weakly up to a subsequence to some $\bar u:\R^n\times\R_-\rightarrow \Sigma$. By writing the equation satisfied by $u_{\rho_k}$ and taking $\rho_k \rightarrow 0$, we see that $\bar u(x,t)$ solves $ \bar u_t - \Delta \bar u = 0$ in the sense of distributions. 

We now demonstrate the parabolic homogeneity of $\bar u$. Note that, for all $R>0$,
\begin{equation*}
N(\bar u,0,0,R) = \lim_{\rho_k\rightarrow 0} N(u,x_0,t_0,\rho_kR) = N(u,x_0,t_0)\geq 1.
\end{equation*}
That is, $N(\bar u,R)$ is constant in $R$. Looking at the calculations for $u^\eps$ in the proof of Lemma \ref{lem:mono}, we can integrate (\ref{eq:IR}) in $R$ and pass to the limit as $\eps\rightarrow 0$. (As above, this can be justified using the strong $H_{\loc}^1(\Omega\times\R_+)$-convergence of $u^\eps$.) We obtain
\[ \int_{R_1}^{R_2}I(u,R)\,dR = -\int_{R_1}^{R_2}R^2\iR u\cdot\left(u_t-\lambda(t)\circ u + \nabla u\cdot\frac{x}{2t} \right)G\,dx\,dR.\]
Evaluating this expression for $u_{\rho_k}(x,t)=u(\rho_k x,\rho_k^2 t)$ and taking the limit as $u_{\rho_k}\rightharpoonup \bar u$ in $H^1(\R^n\times\R_-)$, we conclude
\begin{equation}\label{eq:I}
\int_{R_1}^{R_2}I(\bar u,R)\,dR = -\int_{R_1}^{R_2}R^2\iR \bar u\cdot\left(\bar u_t + \nabla \bar u\cdot\frac{x}{2t} \right)G\,dx\,dR.
\end{equation}
Applying the same procedure to (\ref{eq:Hpr}) and (\ref{eq:Ip}), we obtain
\begin{align}
\int_{R_1}^{R_2}\frac{d}{dR}H(\bar u,R)\,dR &= -\int_{R_1}^{R_2}4R\iR \bar u\cdot\left(\bar u_t+\nabla \bar u\cdot\frac{x}{2t}\right)G\,dx\,dR,\label{eq:dH}\\
\int_{R_1}^{R_2}\frac{d}{dR}I(\bar u,R)\,dR &= \int_{R_1}^{R_2}4R^3\iR \left|\bar u_t+\nabla \bar u\cdot\frac{x}{2t}\right|^2 G\,dx\,dR,\label{eq:dI}
\end{align}
for any $0<R_1<R_2$. We want to show $\bar u(\rho x,\rho^2 t) = \rho^\alpha\bar u(x,t)$, with $\alpha = N(u,x_0,t_0)$, i.e. that $\dfrac{d}{d\rho} \dfrac{\bar u(\rho x,\rho^2 t)}{\rho^\alpha} = 0$ for all $\rho>0$. By a rescaling, it suffices to demonstrate that $\dfrac{d}{d\rho}\Big|_{\rho=1} \dfrac{\bar u(\rho x,\rho^2 t)}{\rho^\alpha} = 2t\bar u_t+x\cdot\nabla \bar u -\alpha \bar u = 0$. Using (\ref{eq:I}), (\ref{eq:dH}), and (\ref{eq:dI}), we compute
\begin{align}\label{eq:intalpha}
\int_{R_1}^{R_2}\iR  |2t \bar u_t+&\nabla \bar u\cdot x - \alpha \bar u|^2 G\,dx\,dR\nonumber\\
  &= \int_{R_1}^{R_2}\left(R\frac{dI(\bar u,R)}{dR} - 4\alpha I(\bar u, R) + \alpha^2 H(\bar u,R)\right)dR.
\end{align}
Since $N(R) = N(\bar u,R)$ is constant, we have $N'(R)H(R) = I'(R)-\alpha H'(R)/2=0$. Thus, we conclude from (\ref{eq:I}) and (\ref{eq:dH}) that
\begin{equation*}
\int_{R_1}^{R_2}R\frac{d}{dR}I(\bar u,R)\,dR = \frac{\alpha}{2}\int_{R_1}^{R_2}R\frac{d}{dR}H(\bar u,R)\,dR = \int_{R_1}^{R_2}2\alpha I(\bar u,R)\,dR.
\end{equation*}
Because $\alpha^2H(R) = 2\alpha I(R)$, (\ref{eq:intalpha}) implies  $2t\bar u_t+x\cdot\nabla \bar u -\alpha \bar u = 0$.

 \end{proof}

The function $\bar u:\Omega\times\R_-\rightarrow\Sigma $, which depends on the base point $(x_0,t_0)$, is called a \emph{blow-up} or \emph{tangent map} of $u$. We are now ready to show a gap in the values of $N(u,x_0,t_0)$. 
 
 \begin{lemma}\label{lem:gap}
Let $(x_0,t_0)\in \Gamma$.  Then either $N(u,x_0,t_0) = 1$ or $N(u,x_0,t_0)\geq 1+\delta_n$ for some constant $\delta_n$ depending on the dimension.
 \end{lemma}
 \begin{proof} 
Let $\bar u:\R^n\times\R_-\rightarrow\Sigma$ be a blow-up with base point $(x_0,t_0)$ and homogeneity $\alpha=N(u,x_0,t_0)\geq 1$. The image of $\bar u$ lies in $d$ lines of $\Sigma$ for some $1\leq d\leq m$.

First suppose $d=1$. In this case, $\bar u$ is a scalar-valued caloric function, and the parabolic frequency $\alpha=N(u,x_0,t_0)$ is either 1 or at least 2.

If $d=2$, we may assume that $\bar u=(\bar u_1,\bar u_2,0,\ldots,0)$. Let $A_i\subset\R^n$ be the set where $\bar u_i>0$, $i=1,2$. (It is clear that $A_i$ will not change with $t$.) Define a function $U$ on $\R^n\times\R_-$ by
\begin{equation*}
U(x,t)= \begin{cases}
		\bar u_1, &x\in A_1\\
		-\bar u_2, &x\in A_2.
		\end{cases}
\end{equation*}
The scalar function $U$ solves $U_t=\Delta U$ in the interiors of $A_1$ and $A_2$ because $\bar u_1$ and $\bar u_2$ do, respectively.  Because the image $\bar u(\R^n\times \R_-)$ lies in two half-line segments, we see from the weak form of $\bar u_t - \Delta \bar u = 0$, with an appropriate test function, that $\left|\dfrac{\partial\bar u_1}{\partial\nu}\right|=\left|\dfrac{\partial\bar u_2}{\partial\nu}\right|$ on $\partial A_1\cap\partial A_2$, where $\nu$ is the unit normal vector to $\partial A_1\cap\partial A_2$.  Hence $U$ is a caloric function, and we conclude either $\alpha=1$ or $\alpha\geq2$.

Finally, if $d\geq3$, then as above we may assume $\bar u=(\bar u_1,\ldots,\bar u_d,0,\ldots,0)$. We write $\bar u(x,t) = \bar u(\rho\xi,\rho^2\tau) = \rho^\alpha w(\xi,\tau)$, where $w(\xi,\tau)$ maps from $\mathbb S^{n-1}\times\R_-$ into $\Sigma$. Since $\bar u_t-\Delta \bar u=0$, we see by an easy computation
\begin{equation}\label{eq:sphere}
w_\tau-\Delta_\xi w = c(n,\alpha) w \text{ in }\mathbb{S}^{n-1}\times\R_-,
%\frac{\partial}{\partial \tau}w_j-\Delta_\xi w_j = c(n,\alpha) w_j\text{ in }\{w_j>0\}\subset \mathbb{S}^{n-1}\times\R_-,
\end{equation}
with $c(n,\alpha)=\alpha(\alpha+n-2)$. Define $B_j=\{w_j>0\}\subset\mathbb S^{n-1}$, $j=1,\ldots,m$.

Consider the quantity $\prod_{j=1}^d(\int_{\mathbb{S}^{n-1}}w_j^2)$.  Differentiating in $\tau$ and using (\ref{eq:sphere}), we have
\begin{equation}\label{eq:log}
 \frac{1}{2d}\frac{\partial}{\partial\tau}\log\prod_{j=1}^d\left(\int_{\mathbb{S}^{n-1}}w_j^2\right) = -\frac{1}{d}\sum_{j=1}^d \frac{\int_{\mathbb{S}^{n-1}}|\nabla w_j|^2}{\int_{\mathbb{S}^{n-1}}w_j^2} + c(n,\alpha).
\end{equation}
Since $d\geq 3$, $B_{j_*}$ has volume $\leq \frac{1}{3}|\mathbb{S}^{n-1}|$ for some $j_*$.  A simple eigenvalue estimate, using spherical symmetrization and monotonicity of the first eigenvalue with respect to the domain, gives $\lambda_1(B_{j_*})\geq n-1+\eta_n$, from which we conclude
\[\frac{1}{d}\sum_{j=1}^d \lambda_1(B_j) 
%\geq \min_j \lambda_1(\{\Psi_j>0\})
\geq n-1+\frac{\eta_n}{d}.\]
If $c(n,\alpha) < n-1+\dfrac{\eta_n}{2d}$, then the right-hand side of (\ref{eq:log}) is uniformly negative, which implies
\[\prod_{j=1}^d\int_{\mathbb{S}^{n-1}}w_j^2(\cdot,\tau)\leq c_1 e^{-c_2\tau}.\]
However, as a solution of (\ref{eq:sphere}) with $c(n,\alpha)>0$, $w(\xi,\tau)$ grows exponentially in $\tau$. We conclude $c(n,\alpha) = \alpha(\alpha+n-2)\geq n-1+\dfrac{\eta_n}{6}$, or $\alpha\geq 1+\delta_n$.%($\lambda=\alpha(\alpha+n-2)\geq n-1$).

\end{proof}

Define $\Gamma^* \subset \Gamma$ by $\Gamma^* = \{(x_0,t_0)\in\Gamma : N(x_0,t_0) = 1\}$.  Let $S = \Gamma \setminus \Gamma^*$. We see that $S$ is closed in $\Omega$ because $N(x,t)\geq 1+\delta_n$ on $S$ by Lemma \ref{lem:gap}, and $N(x,t)$ is upper-semicontinuous in $(x,t)$.

Our next step is to show that the singular set $S$ is of parabolic Hausdorff dimension at most $n$. For any $R>0$, denote the parabolic cylinders $Q_R(x,t)=B_R(x)\times[t-R^2,t]$ and $Q_R=Q_R(0,0)$. Recall that the $s$-dimensional parabolic Hausdorff measure of a set $E\subset\R^n\times\R$ is defined by
\[\mathcal P_s(E) = \sup_{\delta>0}\,\, \inf \left\{\sum_{i=1}^\infty R_i^s \,\Big|\, E\subseteq \bigcup_{i=1}^\infty Q_{R_i}(x_i,t_i), R_i \leq \delta\right\}.\]
The parabolic Hausdorff dimension of $E$ is given by $\dim_P(E)=\inf\{s|\mathcal P_s(E)=0\}$. 
Since $\mathcal{P}^{n+2}$ is equivalent to the Lebesgue measure $\mathcal{L}^{n+1}$ on subsets of $\R^n\times\R$, Theorem \ref{thm:dim} identifies $S$ as a codimension-2 set.

We now state the version of Federer's Reduction Principle that we will need in the proof of Theorem \ref{thm:dim}. It is a straightforward generalization to the parabolic setting of the version stated in \cite[page~49]{L} for Euclidean Hausdorff dimension.

\begin{theorem}\label{thm:fed}
Let $\mathcal F\subset L_{\loc}^1(\R^n\times\R_+,\R^m)$ be a family of functions satisfying the following properties:
\begin{enumerate}
\item[(H1)] (Closure under translation and scaling) If $\phi\in\mathcal F$ and $(x,t)\in Q_{1-\rho}$, $0<\rho<1$, then $\phi^{x,t,\rho}\in \mathcal F$, where $\phi^{x,t,\rho}(y,s) = \phi(x+\rho y,t+\rho^2 s)$.
\item[(H2)] (Existence of homogeneous blow-ups) For any $(x_0,t_0)\in Q_1$, $\rho_k\rightarrow 0$, and $\phi\in\mathcal F$, then $\phi^{0,0,\rho_k} \rightharpoonup \psi$ in $\mathcal D ' (\R^n\times \R)$ for some $\psi\in\mathcal F$, up to a subsequence.
\item[(H3)] (Singular set hypothesis) For each $\phi\in\mathcal F$, there exists a closed set $S_\phi\subset \R^n\times\R$ such that
\begin{enumerate}
\item[(a)] For $(x,t)\in Q_{1-\rho}$, $0<\rho<1$, then $S_{\phi^{x,t,\rho}} = P_{x,t,\rho}(S_\phi)$, where the parabolic dilation 
 \[ P_{x,t,\rho}(y,s)=\left(\frac{y-x}{\rho},\frac{s-t}{\rho^2}\right).\]
\item[(b)] If $\phi_k,\phi\in\mathcal F$ such that $\phi_k\rightharpoonup\phi$ in $\mathcal D ' (\R^n\times \R)$, then $S_{\phi_k}$ converges to $S_{\phi}$ in the following sense: for any $\eps>0$, there exists and integer $K$ such that if $k\geq K$, 
\[S_{\phi_k}\cap Q_1 \subset \{(x,t)\in \R^n\times \R \,|\, d_P((x,t),S_\phi)<\eps\},\]
where $d_P$ is the parabolic distance.
\end{enumerate}
\end{enumerate}
Then, either $S_\phi\cap Q_1=\null$ for every $\phi\in\mathcal F$, or there exists an integer $d\in[0,n+1]$ such that $\dim_P(S_\phi\cap Q_1)\leq d$ for all $\phi\in\mathcal F$. In the latter case, there exists a $\psi\in\mathcal F$ and a ``hyperplane'' $L\subset \R^n\times\R_+$ (i.e. a set such that $P_{0,0,\rho}(L)=L$ for all $\rho>0$) with $\dim_P(L) = d$, such that $\psi^{x,t,\rho}=\psi$ for every $(x,t)\in L$ and $\rho>0$, and $S_\psi=L$.
\end{theorem}

 \begin{theorem}\label{thm:dim}
  The parabolic Hausdorff dimension $\dim_P(S)\leq n$.
 \end{theorem}
 \begin{proof}
We will focus on $S\cap Q_1(x_0,t_0)$ for $(x_0,t_0)\in S$. Assume $(x_0,t_0) = (0,0)$. For $N_0>0$, define the family $\mathcal{A}_{N_0}$ as follows: a pair $(u,\lambda)$, where $u\in H^1(\Omega\times [-1,0],\Sigma)$ and $\lambda\in L^\infty([-1,0], \R^m)$, is a member of $\mathcal A_{N_0}$ if
\begin{equation*}
\begin{cases}
u_j(\partial_t u_j - \Delta u_j) = \lambda_j(t)u_j^2  &\mbox{in } \mathcal{D}'(\Omega\times [-1,0]),\\
\|\lambda_j(t)\|_{L^\infty}\leq N_0,  &j=1,\ldots m.\\
N(u,0,0,1)\leq N_0
\end{cases}
\end{equation*}
Note that we do not require $\lambda_j(t)=\int_\Omega |\nabla u_j|^2\,dx$ in the definition of $\mathcal{A}_{N_0}$, but if $u$ is a suitable weak solution of (\ref{eq:weak}) restricted to $t\in [-1,0]$ and $\lambda_j(t) = \int_\Omega |\nabla u_j|^2\,dx$, then $(u,\lambda)\in\mathcal{A}_{N_0}$ for $N_0$ large enough.  

This family is compact in the following sense.  For a sequence $\{(u^i,\lambda^i)\}\in\mathcal{A}_{N_0}$, there is a subsequence $\{(u^{i'},\lambda^{i'})\}$ such that $u^{i'}\rightarrow u^*$ weakly in $H_{\loc}^1(\Omega\times [-1,0],\Sigma)$, $\lambda^{i'}\rightarrow\lambda^*$ weak-$*$ in $L^\infty([-1,0],\R^m)$, and we have $u_j^*(\partial_t u_j^* - \Delta u_j^*) = \lambda_j^*(t)(u_j^*)^2$ in $\mathcal D'(\Omega\times [-1,0])$ with $N(u^*,0,0,1)$ and $\|\lambda^*\|_{L^\infty}\leq N_0$.  Since $\{u^i\}$ are locally Lipschitz in the $x$ variable uniformly in $i$, $u^*$ is also Lipschitz in $x$.

Consider the family $\mathcal{F} = \{u | (u,\lambda)\in\mathcal A_{N_0} \mbox{ for some }\lambda\}$. (H1) and (H2) in Theorem \ref{thm:fed} are easy to verify for $\mathcal F$, in light of Lemma \ref{lem:blowup}. 

For $u\in \mathcal F$, let $S_u=\{(x,t)\in Q_1 \,|\, N(u,x,t)\geq 1+~\delta_n\}$. To verify (a) of (H3), take $(u,\lambda)\in \mathcal{A}_{N_0}$ and consider $u^{x,t,\rho}(y,s)=u(x+\rho y,t+\rho^2 s)$.  If $(x,t)\in Q_R$, then $u^{x,t,\rho}$ solves our equation with a new $\lambda$, say $\lambda^\rho$. Since $0<\rho<1$, $(u^\rho,\lambda^\rho)$ belongs to $\mathcal{A}_{N_0}$.  Clearly, the parabolic dilation $P_{x,t,\rho}(S_u)\subset S_{u^\rho}.$ Condition (b) of (H3) follows from the compactness of solutions described above.

Theorem \ref{thm:fed} implies $\dim_P(S_u)\leq d$ for all $u\in\mathcal{F}$, for some integer $d\in [0,n+1]$.  Assume that $d=n+1$. Then Theorem \ref{thm:fed} gives the existence of a $\psi\in \mathcal F$ with $S_\psi=L$, a ``hyperplane'' of parabolic Hausdorff dimension $n+1$. From the invariance of $L$ under parabolic dilation, we see that $L$ is a hypersurface of Euclidean Hausdorff dimension $n$. Since the parabolic vanishing order of $\psi$ is strictly greater than $1$ at each point of $S_\psi$, we have that $\nabla \psi=0$ and $\psi=0$ on an analytic hypersurface of Euclidean dimension $\geq n$. Uniqueness of the Cauchy problem implies $\psi\equiv0$ and $\nabla\psi\equiv 0$ in an open neighborhood of $S_\psi$, i.e. in a set with parabolic Hausdorff dimension $n+2$, a contradiction.

Thus, $\dim_P(S_u\cap Q_1)\leq n$ for every $u\in\mathcal F$, in particular for $u$ a suitable weak solution of (\ref{eq:weak}).

 \end{proof}

Our next goal is to characterize $\Gamma^*=\{(x,t)\in\Omega\times\R_+|u(x,t)=0,N(u,x,t)=1\}$ as a $C^{2,\beta}$ surface. Note that if $\nabla u(x_0,t_0)=0$ at a point $(x_0,t_0)\in\Gamma$, then $N(x_0,t_0)>1$ and $(x_0,t_0)\in S$ by Lemma \ref{lem:gap}. This rules out the possibility of $\Gamma^*$ having a horizontal tangent.

We now examine the behavior of $u$ near a point $(x_0,t_0)$ of $\Gamma^*$. (We may assume $(x_0,t_0)=(0,0)$ and that $u$ is defined on $\Omega\times [-2,\infty)$.)  For a vector $\vec{e}\in\mathbb{S}^{n-1}$, let $T(\vec{e},\eps)$ be the following subset of $Q_1$:
\[T(\vec{e},\eps) = \left\{(x,t)\in Q_1\,\Big|\, |x\cdot\vec{e}|\leq\eps\right\},\]
i.e. the Cartesian product of the $\eps$-strip in the $\vec{e}$ direction passing $x=0$, with the time interval $[-1,0]$. Additionally, define two subsets $Q_+^e(0,0)$ and $Q_-^e(0,0)$ by
\[Q_\pm^e(0,0) = \left\{x\in Q_1\,|\,x\cdot\vec{e}\gtrless 0\right\}.\]

\begin{lemma}\label{lem:flat}
Let $u: \Omega\times [-2,\infty)\rightarrow\Sigma$ be a suitable weak solution as in \textup{(\ref{eq:weak})}, and such that $(0,0)\in\Gamma^*$ and $H(u,0,0,2)=1$. Then for every $\eps\in(0,1)$, there is a $\delta(\eps)>0$ such that if $N(u,0,0,2)\leq 1+\delta(\eps)$, then $u$ can be well-approximated by functions linear in $x$ and constant in $t$ in $Q_1\setminus T(\vec{e},\eps)$, for some $\vec{e}\in\mathbb{S}^{n-1}$.  By this we mean:
\begin{align*}
& \|u(x,t)-\ell_\eps^\pm(x)\|\leq\eps \quad \mbox {in } C^{0,1}(Q_\pm^e(0,0)\setminus T(\vec{e},\eps)),\\
& \ell_\eps^\pm(x)= a_\eps^\pm (x\cdot\vec{e}), \quad a_\eps^\pm\in\Sigma \mbox{ with } 1\leq d_\Sigma(a_\eps^\pm,0)\leq 2.
\end{align*}
\end{lemma}
\begin{proof}
The conclusion follows from a contradiction argument. If for some $\eps\in(0,1)$ there exists $u:\Omega\times [-2,\infty)\rightarrow\Sigma$ with $N(u,0,0,2) = 1$, then by the continuity of $N(u,x,t,2)$ in $x$ and Lemma \ref{lem:gap}, $u$ satisfies $u(x_0+\rho x,t_0+\rho^2 t)=\rho u(x_0+x,t_0+t)$ for any $(x_0,t_0)\in \Gamma\cap Q_1$, $\rho>0$. But such a $u$ must be given by a two-component function linear in $x$ and constant in $t$.

\end{proof}

We can extend this flatness to a neighborhood of $(0,0)$ in $\Gamma^*$ as follows: Let $u:\Omega\times [-2,\infty)\rightarrow \Sigma$ be a weak solution as above such that $N(0,0,2)\leq 1+\delta$ and $(0,0)\in\Gamma^*$.  Then for all $(x,t)\in\Gamma\cap Q_{2/3}$, we have $N(x,t,1)\leq 1 + c(\delta)$, $c(\delta)\rightarrow 0$ as $\delta\rightarrow 0$. Thus, for any $0<R<1$, one has $N(x,t,R)\leq 1+c(\delta)$.  This implies that when $\delta$ is sufficiently small, at every point $(x,t)\in\Gamma\cap Q_{2/3}$, and for every scale $0<\rho<1$, a properly scaled version of $u$ can be well-approximated by two linear functions $\ell_\rho^\pm(x)$ in $Q_\rho(x,t) \setminus \{(x,t) + T(e_\rho,\eps \rho)\}$.

Next we prove a clean-up lemma stating that locally, near $\Gamma^*$, $u$ maps into two components of $\Sigma$ whose supports are separated by $\Gamma^*$.
\begin{lemma}\label{lem:clean} 
Let $u$ be as in Lemma \textup{\ref{lem:flat}} and let $(x_0,t_0)\in\Gamma^*$.  Then $u$ maps $Q_1(x_0,t_0)$ into at most two different components of $\Sigma$.
\end{lemma}
\begin{proof}
 We may assume $(x_0,t_0)=(0,0)$. Since $N(0,0,2)\leq 1+\delta$, Lemma \ref{lem:flat} implies that there are exactly two lines $L_1,L_2$ in $\Sigma$ such that $u\left(Q_{3/2}\setminus T(\vec{e},\eps)\right)\subseteq L_1\cup L_2$, for some $\eps=\eps(\delta)\rightarrow 0+$ as $\delta\rightarrow 0+$. Moreover, for every $(x,t)\in \Gamma\cap Q_1$, one has $u\left(Q_{1/2}(x,t)~\setminus~T(\vec{e},\eps)\right)\subseteq L_1\cup L_2$.
 
Thus, to prove the lemma, it suffices to show:  If $(0,0)\in\Gamma^*$, $N(0,0,1/2)\leq 1+\tilde{\delta}$, for sufficiently small $\tilde{\delta}$, and $u(Q_{1/2}\setminus T(\vec{e},\eps))\subseteq L_1\cup L_2$ for some $\eps=\eps(\tilde{\delta})$ and  some vector $\vec{e}\in\mathbb{S}^{n-1}$, where $L_1$ and $L_2$ are two distinct lines in $\Sigma$, then for any $r\in(0,1/2)$,
\[u\left(Q_r\setminus T(\vec e_r,\eps r)\right)\subseteq L_1\cup L_2\]
for the same two lines $L_1$ and $L_2$ in $\Sigma$, and for some $\vec e_r$.

The above fact is implied by the following observation.  For each $r\in(0,1/2)$, there are two  lines $L_1^r$ and $L_2^r$ in $\Sigma$, such that $u(Q_r~\setminus~T(\vec e_r,\eps r))\subseteq L_1^r\cup L_2^r$.  However, these $L_1^r$ and $L_2^r$ may depend on $r$.  One easily sees that $L_1^r$ and $L_2^r$ must depend on $r$ continuously.  There are finitely many choices of $L_1^r$ and $L_2^r$ in $\Sigma$, thus the choice of $L_1^r\cup L_2^r$ will be fixed for all $0<r<1$ by continuity.

\end{proof}

 \begin{theorem}\label{thm:Reg}
  Let $u$ be a suitable weak solution of \textup{(\ref{eq:weak})}, and let $(x_0,t_0)\in\Gamma^*$.  Then in a sufficiently small parabolic neighborhood $Q_R(x_0,t_0)$, there are exactly two components $\Omega_1$ and $\Omega_2$ separated by the $C^{2,\beta}$ surface $\partial\Omega_1\cap\partial\Omega_2 = \Gamma\cap Q_R(x_0,t_0)$.
 \end{theorem}
 \begin{proof}
We may rescale and translate $u$ to fit the hypotheses of Lemma \ref{lem:flat}. By Lemma \ref{lem:clean}, we have $u(Q_1)\subseteq L_1\cup L_2$ for two lines $L_1,L_2$ in $\Sigma$.  Define $\Omega_j = u^{-1}(L_j), j=1,2$. By reordering our basis, we may suppose $u_j>0$ on $\Omega_j$, $j=1,2$.
 
Define $u^*:Q_1\rightarrow\R$ by $u^*=u_1$ on $\Omega_1$, $u^*=-u_2$ on $\Omega_2$. By (\ref{eq:weak}), we have that $u_t^*-\Delta u^* = f(u^*)$ for a Lipschitz function $f$. Using again the parabolic estimates in \cite[Chapter~7]{Lie}, we have $u^*\in C^{2,\beta}$ for any $\beta\in(0,1)$. Then, since the vanishing order along $\Gamma^*$ is one, $u^*=0$ defines a $C^{2,\beta}$ hypersurface for any $\beta\in (0,1)$. But the zero set of $u^*$ is $Q_1\cap \Gamma^*$.

 \end{proof}

\section{Long-time behavior of $u(x,t)$}\label{sec:time}
In this section, we characterize the long-time behavior of a suitable weak solution $u$ as a map into $\Sigma$ that is a stationary solution of problem (P*). (See the Introduction.) Throughout this section, we assume the initial condition $g(x)$ of (\ref{eq:weak}) satisfies $\int_\Omega g_j^2\,dx = c_j^2=1$ for each $j$.  Recall from (\ref{eq:weak}) that $u$ satisfies the estimate $\int_0^T\int_\Omega|\partial_t u|^2 dx dt + \int_\Omega|\nabla u|^2 dx \leq \int_\Omega |\nabla g|^2 dx$. Thus, $\lambda_j(t)=\int_\Omega|\nabla u_j(x,t)|^2 dx$ is nonincreasing in $t$.  Since $\lambda_j(t)$ is nonnegative, it has a limit $\lambda_j^\infty=\lim_{t\rightarrow\infty} \lambda_j(t)$.

\begin{theorem}\label{thm:asymptotics}
 $\lim_{t\rightarrow\infty} u(x,t)=u^\infty(x)$ exists in $H_0^1(\Omega, \Sigma)$, and each component $u_j^\infty$ is an eigenfunction of $\Delta$ on $\{u_j^\infty > 0\}$ with eigenvalue $\lambda_j^\infty$.
\end{theorem}
\begin{proof}
 From the integral estimate in (\ref{eq:weak}), we see that the family $\{u(\cdot,t)\}$ is uniformly bounded in $H_0^1(\Omega,\Sigma)$, so we can find a sequence $\{t_i\}$ with $t_i\rightarrow\infty$ such that, for $u_j^{(i)}(x)=u_j(x,t_i)$,
\begin{equation*}
\begin{cases}
       u_j^{(i)}\rightarrow u_j^\infty & \mbox{in } L^2(\Omega,\Sigma), \\
	   u_j^{(i)}\rightharpoonup u_j^\infty & \mbox{in } H_0^1(\Omega,\Sigma), \\
	   \partial_t u_j^{(i)}-[\lambda_j(t_i)-\lambda_j^\infty]u_j^{(i)} \rightarrow 0 & \mbox{in } L^2(\Omega, \Sigma),
\end{cases}
\end{equation*}
as $i\rightarrow\infty$, for all $j=1,\ldots,m$. By the evolution equation of (\ref{eq:weak}), we can see that $\partial_t u_j^{(i)}-[\lambda_j(t_i)-\lambda_j^\infty]u_j^{(i)}=\Delta u_j^{(i)}+\lambda_j^\infty u_j^{(i)}-\nu_j(x,t_i)$ in $\mathcal{D}'(\Omega\times\R_+)$, and since $\nu_j$ is supported in $\{u_j(x,t)=0\}$, we conclude
\begin{equation}\label{eq:eig}
\begin{cases}
 \Delta u_j^\infty + \lambda_j^\infty u_j^\infty = 0 & \mbox{in } \Omega_j,\\
 u_j^\infty = 0 & \mbox{on } \partial\Omega_j,
\end{cases}
\end{equation}
where $\Omega_j = \{u_j^\infty>0\}$. Since $\int_\Omega |u_j^{(i)}|^2\,dx=1$ for all $i$, we have that $ \int_\Omega|u_j^\infty|^2dx = 1$ for each $j=1,\ldots,m$. Thus, we conclude from (\ref{eq:eig}) that $\lambda_j^\infty = \|u_j^\infty\|_{H_0^1(\Omega)}^2$.
Since $\lambda_j(t)\rightarrow\lambda_j^\infty$ for each $j$, we have 
\[\|u^{(t_i)}\|_{H_0^1(\Omega,\Sigma)}^2 = \sum_{j=1}^m \|u_j^{(t_i)}\|_{H_0^1(\Omega)}^2\rightarrow \|u^\infty\|_{H_0^1(\Omega,\Sigma)}^2,\]
and therefore, $u^{(t_i)}\rightarrow u^\infty$ strongly in $H_0^1(\Omega,\Sigma)$.

To show that $u(\cdot,t)\rightarrow u^\infty$ in $H_0^1(\Omega,\Sigma)$ (not only along a sequence $\{t_i\}$), we look at the $L^2$-distance between $u$ and $u^\infty$. Since $u$ and $u^\infty$ map into $\Sigma$, for a given $(x,t)\in \Omega\times\R_+$ at most one component $u_j$ (respectively $u_j^\infty$) is nonzero. For $t\geq 0$, let 
\[\Omega_{jk}(t)=\{x\in \Omega | u_j(x,t)>0, u_k^\infty(x)>0\}.\]
It is easy to check that $\|y_1-y_2\|_{\R^m}^2\leq d_\Sigma^2(y_1,y_2)\leq 2\|y_1-y_2\|_{\R^m}^2$ for any $y_1,y_2\in\Sigma$, so convergence in $L^2(\Omega,\Sigma)$ is equivalent to convergence in $L^2(\Omega,\R^m)$. It is clear that
\begin{equation*}
\|u(x,t)-u^\infty(x)\|_{\R^m}^2 = \begin{cases} |u_j(x,t)-u_k^\infty(x)|^2, &\mbox{in } \Omega_{jj}(t),\\
										 (u_j(x,t))^2+(u_k^\infty(x))^2, &\mbox{in } \Omega_{jk}(t), j\neq k.
					\end{cases}					
\end{equation*}
Thus, the distance in $L^2(\Omega,\R^m)$ can be written
\begin{align*}
D(t) &:= \int_\Omega \|u(x,t)-u^\infty(x)\|_{\R^m}^2\,dx\\
 &= \sum_{j=1}^m \int_{\Omega_{jj}(t)} (u_j - u_j^\infty)^2\,dx + \sum_{j\neq k} \int_{\Omega_{jk}(t)} (u_j^2 + (u_k^\infty)^2)\,dx.
 \end{align*}
Differentiating the terms individually, for $j=1,\ldots,m$ we have
\begin{align}\label{eq:calc1}
\frac{d}{dt}\int_{\Omega_{jj}(t)}(u_j-u_j^\infty)^2\,dx &= \int_{\Omega_{jj}(t)} 2(u_j-u_j^\infty)\partial_t u_j\,dx\nonumber\\
&\quad + \int_{\partial \Omega_{jj}(t)} (u_j-u_j^\infty)^2(\vec v(x,t)\cdot \vec n_{jj}(x,t))\,d\sigma(x)\nonumber\\
&= \int_{\Omega_{jj}(t)}2(u_j - u_j^\infty)(\Delta u_j + \lambda_j(t)u_j)\,dx\nonumber\\ 
&\quad + \int_{\partial \Omega_{jj}(t)} (u_j-u_j^\infty)^2(\vec v(x,t)\cdot \vec n_{jj}(x,t))\,d\sigma(x),
\end{align}
where $\vec v(x,t)$ is the velocity of the moving interface at $(x,t)\in\Gamma$, and $\vec n_{jj}(x,t)$ is the outward unit normal to $\Omega_{jj}(t)$ at $x\in\partial\Omega_{jj}(t)$. We have used (\ref{eq:weak}) and the fact that $\supp(\nu_j)\subset\{u_j=0\}$. Since $u_j=0$ or $u_j^\infty=0$ at every point of $\partial \Omega_{jj}(t)$, we can integrate by parts and write the first integral on the right-hand side as
\begin{equation*}
\int_{\Omega_{jj}(t)} 2(u_j\Delta u_j + \lambda_j(t) u_j^2 + \nabla u_j\cdot \nabla u_j^\infty - \lambda_j(t) u_j u_j^\infty)\,dx.
\end{equation*}
Integrating by parts again and using (\ref{eq:eig}), this expression becomes
\begin{equation*}
\int_{\Omega_{jj}(t)} 2(u_j\Delta u_j + \lambda_j(t) u_j^2 + (\lambda_j^\infty- \lambda_j(t)) u_j u_j^\infty)\,dx.
\end{equation*}
By a similar calculation, for $j\neq k$ we have
\begin{align}\label{eq:calc2}
\frac{d}{dt}\int_{\Omega_{jk}(t)}(u_j^2+(u_k^\infty)^2)\,dx &= \int_{\Omega_{jk}(t)} 2u_j\partial_t u_j\,dx\nonumber\\
&\quad + \int_{\partial \Omega_{jk}(t)} (u_j^2+(u_k^\infty)^2)(\vec v(x,t)\cdot \vec n_{jk}(x,t))\,d\sigma(x)\nonumber\\
&= \int_{\Omega_{jk}(t)} 2(u_j\Delta u_j+\lambda_j(t)u_j^2)\,dx\nonumber\\
&\quad + \int_{\partial \Omega_{jk}(t)} (u_j^2+(u_k^\infty)^2)(\vec v(x,t)\cdot \vec n_{jk}(x,t))\,d\sigma(x).
\end{align}

Now we would like to sum all of the terms (\ref{eq:calc1}) and (\ref{eq:calc2}) to get an expression for $D'(t)$. Looking at the boundary integrals, it is clear that for each $t\geq 0$, each part of the moving boundary (the set $\{u(x,t)=0\}\cup\{u^\infty(x,t)=0\}$) will appear in the sum twice with opposite signs. For example, if $(x,t)\in\partial\Omega_{jk}(t)$, and $u_j(x,t)>0$, then $(x,t)\in\partial\Omega_{j\ell}(t)$ for some $\ell\neq k$, and $\vec n_{jk}(x,t) = - \vec n_{j\ell}(x,t)$. Arguing in this way, we see that all of the boundary integrals in $D'(t)$ cancel. We can therefore write
\begin{align*}
\frac{1}{2}D'(t) &= \sum_{j=1}^m \int_{\Omega_{jj}(t)} (u_j\Delta u_j + \lambda_j(t) u_j^2 + (\lambda_j^\infty- \lambda_j(t)) u_j u_j^\infty)\,dx\\
&\quad + \sum_{j\neq k} \int_{\Omega_{jk}(t)} (u_j\Delta u_j+\lambda_j(t)u_j^2)\,dx\\
&= \sum_{j=1}^m \int_{\{u_j(x,t)>0\}} (u_j\Delta u_j + \lambda_j(t)u_j^2)\,dx + \sum_{j=1}^m(\lambda_j^\infty-\lambda_j(t))\int_{\Omega_{jj}} u_ju_j^\infty \,dx.
\end{align*}
Since $\lambda_j(t)\geq \lambda_j^\infty$, we can drop the second sum, yielding
\begin{align*}
\frac{1}{2}D'(t) \leq \sum_{j=1}^m \left(-\int_\Omega |\nabla u_j|^2\,dx +\lambda_j(t)\int_\Omega u_j^2\,dx\right) = 0.
\end{align*}
 
The convergence of $u_j(\cdot,t_i)$ to $u_j^\infty$ in $L^2(\Omega)$ for each $j$ implies that $D(t_i)\rightarrow 0$. Since $D(t)$ is nonincreasing and converges to zero along a subsequence, we must have $D(t)\rightarrow 0$. By the observation above, this implies $d_\Sigma^2(u(x,t),u^\infty(x))\rightarrow 0$ in $L^2(\Omega,\Sigma)$. Clearly, if any subsequence $u(\cdot, t_k)$ converges in $H_0^1(\Omega,\Sigma)$ as $t_k\rightarrow\infty$, the limit must equal $u^\infty(x)$. The function $u^\infty$ is thus the strong limit of $u(\cdot,t)$ in $H_0^1(\Omega,\Sigma)$.
\end{proof}

For $v\in H_0^1(\Omega,\Sigma)$, define the functional
\[E_\lambda(v) := \int_\Omega |\nabla v|^2 dx - \sum_{j=1}^m \lambda_1(\Omega_j) \int_\Omega v_j^2 dx.\]
It is clear that a solution of problem (P*) is a critical point of $E_\lambda$. 
\begin{theorem}\label{thm:probPstar}
For $u\in H_{\loc}^1(\Omega\times\R_+)$ a suitable weak solution of \textup{(\ref{eq:weak})} with $\int_\Omega g_j^2=1$, the limit $u^\infty(x) = \lim_{t\rightarrow\infty} u(x,t)$ is a stationary solution to problem \textup{(P*)} in the sense that $u^\infty$ is a critical point of $E_\lambda$, in the space $H_0^1(\Omega,\Sigma)$, with respect to both target and domain variations.
\end{theorem}
\begin{proof}
The conclusion is easy to check via a direct computation, in light of Theorem \ref{thm:asymptotics}. Since the variations must map into $\Sigma$, the variations do not alter the eigenvalues $\lambda_1(\Omega_j)$ in $E_\lambda$.
\end{proof}

\end{document}